\documentclass[12pt]{amsart}
\pdfoutput=1
\usepackage{amssymb}
\usepackage{amsmath}
\usepackage{amsfonts}
\usepackage[usenames]{color}
\usepackage{graphicx}
\usepackage{array}
\usepackage{psfrag}
\usepackage{color}
\usepackage{ulem}
\usepackage{fancyhdr}

\usepackage{caption}

\usepackage{import}
\usepackage{xifthen}
\usepackage{pdfpages}
\usepackage{transparent}

\makeatletter
 
 \@addtoreset{equation}{section}
\makeatother

\newtheorem{theorem}{Theorem}
\newtheorem{lemma}[theorem]{Lemma}

\theoremstyle{definition}
\newtheorem{definition}[theorem]{Definition}

\newtheorem{conjecture}[theorem]{Conjecture}

\theoremstyle{remark}

\numberwithin{equation}{section}

\begin{document}

\title{{Simple graphs of order 12 and minimum degree 6 contain $K_6$ minors}}

\date{\today}
\author{
Ryan Odeneal and
Andrei Pavelescu
}

\address{
Department of Mathematics, University of South Alabama, Mobile, AL  36688, USA.
}

\maketitle
\rhead{12verticesmindeg6K6}

\begin{abstract}
We prove that every simple graph of order 12 which has minimum degree 6 contains a $K_6$ minor.
\end{abstract}
\vspace{0.1in}

\section{Introduction}

All the graphs considered in this article are simple (non-oriented, without loops or multiple edges). For a graph $G$, \textit{a minor of G} is any graph that can be obtained from $G$ by a sequence of vertex deletions, edge deletions, and simple edge contractions.
A simple edge contraction means identifying its endpoints, deleting that edge,  and deleting any double edges thus created. A graph $G$ is called \textit{apex} if it has a vertex $v$ such that $G-v$ is planar, where $G-v$ is the subgraph of $G$ obtained by deleting vertex $v$ and all edges of $G$ incident to $v$. In \cite{Jo}, J{\o}rgensen stated the following conjecture:

\begin{conjecture} Let $G$ be 6-connected graph which does not have a $K_6$ minor. Then $G$ is apex. 
\label{Joergensen}
\end{conjecture}
This result relates to Hadwiger's Conjecture \cite{Hugo}, which states:

\begin{conjecture} For every integer $t \ge 1$, if a loopless graph $G$ has no $K_t$ minor, then it is $(t-1)$ colorable.
\label{Hadwiger}
\end{conjecture}

Conjecture \ref{Hadwiger} is known to be true for $t\le 6$. For $t=4$, the Conjecture is equivalent to Appel and Haken's 4-Color Theorem \cite{4color}. For $t=6$, Robertson, Seymour, and Thomas \cite{RST} proved it using a result of Mader. In 1968, Mader \cite{Mader2} proved that a minimal counterexample to Conjecture \ref{Hadwiger} for $t=6$ has to be 6-connected. Together with J{\o}rgensen's Conjecture, it would provide another proof that Conjecture \ref{Hadwiger} holds for $t=6$, along with more information about the structure of graphs with no $K_6$ minors.

J{\o}rgensen himself took steps towards proving Conjecture \ref{Joergensen}. In \cite{Jo}, J{\o}rgensen proved that every graph $G$ with at most 11 vertices and minimal degree $\delta(G)$ at least 6 is contractible to a $K_6$. In his proof, he used the following result Mader \cite{Mader1} proved in 1968:

\begin{theorem} Every simple graph with minimal degree at least 5 either has a minor isomorphic to $K_6^-$ or it has a minor isomorphic to the Icosahedral graph. 
\label{deg5}
\end{theorem} 
The icosahedral graph is the only 5-regular planar graph on 12 vertices.
 Mader also proved the following theorem \cite{Mader1}. 

\begin{theorem} For every integer $2\le t \le 7$ and every simple graph $G$ of order $n\ge t-1$ which has no minor isomorphic to $K_t$, $G$ has at most $(t-2)n-$${t-1}\choose {2}$ edges.
\label{Mader}
\end{theorem} 

Note that for $t=6$, the theorem implies that every graph $G$ of order $n$ and size $4n-9$ or more has a $K_6$ minor. 

In \cite{Jor}, J{\o}rgensen classified the graphs of order $n$ and size $4n-10$.

\begin{theorem} Let $p$ be a natural number, $5\le p \le 7$. Let G be a graph with
$n$ vertices and $(p-2)n$-${p}\choose {2}$ edges that is not contractible to $K_p$. Then
either G is an $MP_{p-5}$-cockade or $p =7$ and $G$ is the complete 4-partite
graph $K_{2,2,2,3}$.
\label{cockade}
\end{theorem}

For $p=6$, this theorem shows that any graph $G$ of order $n$ and size $4n-10$ either contains a $K_6$ minor, or it is a $MP_1$-cockade. The following is J{\o}rgensen's definition of an $MP_1$-cockade.

\begin{definition}
$MP_{1}$-cockades are defined recursively as follows:
  \begin{enumerate}
    \item $K_{5}$ is an $MP_1$-cockade and if $H$ is a 4-connected maximal 
          planar graph then $H*K_1$ is an $MP_1$-cockade. 
    \item Let $G_1$ and $G_2$ be disjoint $MP_1$-cockades, and let 
          $x_1,x_2, x_3$, and $x_{4}$ be the vertices of a $K_4$ subgraph of $G_1$ and let 
          $y_1,y_2, y_3$, and $y_{4}$ be the vertices of a $K_{4}$ subgraph of $G_2.$ Then
          the graph obtained from $G_1 \cup G_2$ by identifying $x_j$ and 
          $y_j,$ for $j=1,2,3,4$, is an $MP_{1}$-cockade. 
  \end{enumerate}
\end{definition}

For two graphs $G_1$ and $G_2$, $G_1\ast G_2$ denotes the graph with vertex set $V(G_1)\sqcup V(G_2)$ and edge set $E(G_1)\sqcup E(G_2)\sqcup E'$, where $E'$ is the set of edges with one endpoint in $V(G_1)$ and the other endpoint in $V(G_2)$. In $G_1\ast v$, we call $v$ \textit{a cone} over $G_1$. A graph $G$ is the \textit{clique sum} of $G_1$ and $G_2$ over $K_p$ if $V(G)=V(G_1)\cup V(G_2)$, $E(G)=E(G_1)\cup E(G_2)$ and the subgraphs induced by $V(G_1)\cap V(G_2)$ in both $G_1$ and $G_2$ are complete of order $p$. In this context, an $MP_1$-cockade is either a cone over a 4-connected maximal planar graph or the clique sum over $K_4$ of two smaller $MP_1$-cockades. 

In \cite{KNTW}, Kawarabayashi,  Norine, Thomas, and Wollan proved that Conjecture \ref{Joergensen} holds for sufficiently large graphs. Little is known about the validity of Conjecture \ref{Joergensen} for small order graphs. In this paper, we prove that J{\o}rgensen's Conjecture holds for graphs of order 12 in a more general setting.

\begin{theorem} Let $G$ be a simple graph of order $12$ and assume that $\delta(G)\ge 6$, where $\delta(G)$ denotes the minimal degree of $G$. Then $G$ contains a $K_6$ minor.
\label{main}
\end{theorem}

Note that the theorem implies J{\o}rgensen's conjecture is vacuously true for graphs of order $12$.

\section{Main Theorem}

For a graph $G$, $V(G)$ denotes its vertex set and $E(G)$ denotes its edge set. The size of $V(G)$ is called \textit{the order} of $G$, and the cardinality of $E(G)$  is called \textit{the size} of $G$. For $n\ge 1$, $K_n$ denotes the complete graph of order $n$ and $K_n^-$ denotes the complete graph of order $n$ with one edge removed. If $v_1, v_2, \ldots,v_k$ are vertices of $G$, then $\big<v_1, v_2, \ldots,v_k\big>_G$ denotes the subgraph of $G$ induced on these vertices. If $v$ is a vertex of $G$, $N(v)$ is the subgraph of $G$ induced by $v$ and the vertices adjacent to $v$ in $G$. If $S$ is a subset of $V(G)$, $G-S$ is the subgraph of $G$ obtained by deleting all of the vertices in $S$ and all the edges of $G$ that $S$ is incident to.\\

The following lemma is a corollary of Theorem \ref{Mader}.

\begin{lemma} Let $G$ be a simple graph of order $n$ and size $4n-10$. If $G-v$ is planar, then $v$ cones over $G-v$.
\label{cone}
\end{lemma}
\begin{proof} Since $G-v$ is planar of order $n-1$, it has at most $3(n-1)-6=3n-9$ edges. This implies that $v$ has at least $4n-10-(3n-9)=n-1$ neighbors, and the conclusion follows.
\end{proof}
Let $G$ denote a simple graph of order 12 and minimal degree $\delta(G)$ at least six. It follows that $G$ has at least size 36. By Theorem \ref{Mader}, if the size of $G$ is at least 39, then $G$ contains a $K_6$ minor. We shall prove Theorem \ref{main} by considering the size of $G$, $36 \le |E(G)|\le 38$. 

\begin{proof}
{\bf Case 1.} Assume $|E(G)|=38$. By Theorem \ref{cockade}, either $G$ contains a $K_6$ minor, or $G$ is apex, or $G$ is the clique sum over $K_4$ of two $MP_1$ cockades.

If $G$ is isomorphic to $H\ast K_1$, where $H$ is a maximal planar graph on 11 vertices, then $\delta(H)\ge 5$ and, by Theorem \ref{deg5}, it follows that $H$ has a $K_6^-$ minor and thus $G$ has a $K_6$ minor.

Assume that $G$ is the clique sum over $S\simeq K_4$ of two $MP_1$-cockades. If $G-S$ has more than two connected components, $Q_1,Q_2,...$, then at least one of them, say $Q_1$, has at most two vertices. But this contradicts the fact that $\delta(G)\ge 6$. So $G-S= Q_1\sqcup Q_2$. Furthermore, unless $|Q_1|=|Q_2|=4$, the graph either contains a $K_7$ subgraph (if $|Q_1|$=3), or $\delta(G)<6$ (if $1\le |Q_1|\le 2$). 
Since $\delta(G)\ge 6$, it follows that each vertex of $Q_i$ connects to at least two other vertices of $Q_i$, for $i=1,2$ respectively. 

Without loss of generality, let $Q_1=\big<v_1,v_2,v_3,v_4\big>_G$, $S=\big<v_5,v_6,v_7,v_8\big>_G\simeq K_4$, and $Q_2=\big<v_9,v_{10},v_{11},v_{12}\big>_G$. If $Q_i$ is not isomorphic to $K_4$ for any of $i=1,2$, say $v_1v_2\notin E(Q_1)$, since $v_1$ and $v_2$ must both connect to both $v_3$ and $v_4$, then contracting the edges $v_1v_3$ and $v_2v_4$ produces a minor of $G$ which contains a $K_6$ subgraph induced by $v_1$, $v_2$, and the four vertices of $S$. If, on the other hand, $Q_1 \simeq Q_2\simeq K_4$, as $\delta(G)\ge 6$, it follows that there are at least $12$ edges between each of the $Q_i$s and $S$. That would imply that $E(|G)|\ge 6+12+6+12+6=42$, a contradiction. It follows that for $|E(G)|=38$, $G$ has a $K_6$ minor.  \\

{\bf Case 2.} Assume that $|E(G)|=37$. Since $\delta(G)\ge 6$, it follows that the degree sequence of $G$ is either $(6,6,6,6,6,6,6,6,6,6,6,8)$ or $(6,6,6,6,6,6,6,6,6,6,7,7)$. In either of the situations, we shall need the following lemma:

\begin{lemma} Let $M$ denote a graph of order 11 and size 34, such that $\delta(M)\ge 5$. Assume that $M$  is not apex and has at most four vertices of degree 5. Then $M$ contains a $K_6$ minor.
\label{11cockade}
\end{lemma}
\begin{proof}

By Theorem \ref{cockade}, either $M$ contains a $K_6$ minor or is a $MP_1$-cockade. Since $M$ is not apex, it follows that $M$ is the clique sum over $S\simeq K_4$ of two $MP_1$-cockades. If $M-S$ has more than two connected components, $Q_1,Q_2,...$, then at least one of them, say $Q_1$, has at most two vertices. As $|Q_1|=1$ would violate the condition $\delta(M)\ge 5$, it follows that $|Q_1|=2$ and the subgraph of $M$ induced by $Q_1$ and $S$ forms a $K_6$. So $M-S= Q_1\sqcup Q_2$ and, without loss of generality, $Q_1=\big<v_1,v_2,v_3\big>_M$, $S=\big<v_4,v_5,v_6,v_7\big>_M\simeq K_4$, and $Q_2=\big<v_8,v_{9},v_{10},v_{11}\big>_M$.
Unless $Q_1 \simeq K_3$, since $\delta(M)\ge 5$, it follows that at least two vertices of $Q_1$ connect to all the vertices of $S$ and thus, via an edge contraction, they induce a $K_6$ minor of $M$. 

If $Q_1\simeq K_3$, there have to be exactly 9 edges connecting the vertices of $Q_1$ to those of  $S$. If there are more than nine, the subgraph induced by the vertices of $Q_1$ and $S$ has 7 vertices and more than $3+9+6=18$ edges, thus it contains a $K_6$ minor by Theorem \ref{Mader}. If there are less than 9, then at least one of the vertices of $Q_1$ has degree less than 5. So all the vertices of $Q_1$ have degree 5, and the subgraph induced by the vertices of $Q_1$ and $S$ has exactly 18 edges. If $L$ denotes the set of edges connecting the vertices of $Q_2$ to the vertices of $S$, then $|L|+|E(Q_2)|=34-18=16$. On the other hand, since $Q_2$ can have at most one vertex of degree 5 in $M$, it follows that $|L|+2|E(Q_2)|\ge 6+6+6+5=23$. Subtracting the last two equalities we get $|E(Q_2)|\ge 7$, a contradiction as $Q_2$ has 4 vertices.
\end{proof}

Assume the vertex degree sequence for $G$ is $(6,6,6,6,6,6,6,6,6,6,6,8)$. Furthermore, without loss of generality, we may assume that $deg_G(v_1)=6$, $deg_G(v_8)=8$, and that $N_G(v_1)=\{v_1,...,v_7\}$. Let $N=\big<v_2,v_3,...,v_7\big>_G$, $H=\big<v_8,...,v_{12}\big>_G$, and let $L$ denote the set of edges of $G$ with one endpoint in $N$ and the other in $H$. The Handshaking Lemma provides the following relations between the sizes of $E(N), L,$ and $E(H)$:
\[2|E(N)|+|L|=30,\,\,\,\,2|E(H)|+|L|=32.\] 

If $|E(H)|=10$, that is $H\simeq K_5$, as every vertex of $H$ must have at least two neighbors in $N$ ($\delta(G)\ge 6$), contracting the edges $v_1v_i$, for $2\le i \le 7$, produces a $K_6$ minor of $G$.

If $|E(H)|\le 9$, then $|L|\ge 14$ and thus $|E(N)|\le 8$. It follows that there is a vertex of $N$, say $v_2$, such that $deg_N(v_2)\le 2$. If $deg_N(v_2)<2$, contracting the edge $v_1v_2$ would produce a minor of $G$ of order 11 and size at least $35$, which would contain a $K_6$ minor by Theorem \ref{Mader}. 

If $deg_N(v_2)=2$, then contracting the edge $v_1v_2$ would produce a minor $M$ of $G$ of order 11 and size precisely $34$. Furthermore, since $v_2$ neighbors exactly 3 vertices of $H$, the maximum degree of $M$ is 8, so it cannot be apex, according to Lemma \ref{cone}. Lastly,  $M$ has at most two vertices of degree 5, since $deg_N(v_2)=2$. By Lemma \ref{11cockade}, $M$ has a $K_6$ minor, and therefore so does $G$.\\

Assume the vertex degree sequence for $G$ is $(6,6,6,6,6,6,6,6,6,6,7,7)$. If the degree 7 vertices are connected in $G$, deleting the edge connecting them would produce a 6 regular subgraph of order 36, to be dealt with in the last case of the proof. So, without loss of generality, assume that $deg_G(v_1)=6$, $deg_G(v_8)=7$, and $N_G(v_1)=\{v_1,...,v_7\}$. Let $N=\big<v_2,v_3,...,v_7\big>_G$, $H=\big<v_8,...,v_{12}\big>_G$, and let $L$ denote the set of edges of $G$ with one endpoint in $N$ and the other in $H$. If $deg(v_i)=7$, for some $9\le i \le 12$, the same argument as before shows that $|E(N)|\le 8$ and thus $G$ contains a $K_6$ minor. So we may assume $deg_G(v_7)=7$ and $v_7v_8\notin E(G).$ Using the Handshaking Lemma, we get:
\[2|E(N)|+|L|=31,\,\,\,\,2|E(H)|+|L|=31.\]

If $|E(H)|=10$, contracting the edges $v_1v_i$, for $2\le i \le 7$, produces a $K_6$ minor of $G$.

If $|E(H)|\le 9$, then $|L|\ge 13$ and thus $|E(N)|\le 9$. If $|E(N)|\le 8$, it follows that there is a vertex of $N$, $v_i$, such that $deg_N(v_i)\le 2$.  If $deg_N(v_i)<2$, contracting the edge $v_1v_i$ would produce a minor of $G$ of order 11 and size at least $35$, which would contain a $K_6$ minor by Theorem \ref{Mader}. 

Assume $deg_N(v_i)=2$. Contracting the the edge $v_1v_i$ produces a minor $M$ of $G$ of order $11$ and size 34. Moreover, since for $2\le j \le 7$,  $v_j$ neighbors at most for of the vertices of $H$, $M$ cannot be apex. Lastly,  $M$ has at most two vertices of degree 5, since $deg_N(v_i)=2$. By Lemma \ref{11cockade}, $M$ has a $K_6$ minor, and therefore so does $G$.

It follows that $N$ is 3-regular, $L=13$ and that $|E(H)|=9$; that is, $H\simeq K_5^-$. If the missing edge of $H$ has $v_8$ as its endpoint, and since $deg_G(v_8)=7$,  it follows that $v_8$ neighbors at least 4 vertices of $N$. As the other endpoint, say $v_9$, neighbors at least 3 vertices of $N$, it follows that there exists $2\le i \le 6$ such that $v_i$ is a common neighbor of $v_8$ and $v_9$. Contracting the edges $v_iv_8$ and $v_1v_j$, for $2\le j \le 7, j\ne i$, one obtains a $K_6$ minor of G, as every vertex of $H$ neighbors at least one of the $v_j$s. 

If the missing edge of $H$ is $v_9v_{10}$, as $N$ is 3-regular, $deg_G(v_7)=7$ and $v_7$ does not neighbor $v_8$, it follows that $v_7$ neighbors $v_j$, for all $9\le j\le 12$. Since every vertex of $H$ neighbors at least two vertices of $N$, if follows that contracting the edges $v_7v_9$ and $v_1v_i$, for $2\le i \le 6$, produces a $K_6$ minor of $G$.\\
 
{\bf Case 3.} Assume $|G|=36$, that is $G$ is a 6-regular graph. Let $v_2,v_3,...,v_7$ be the neighbors of some vertex $v_1$ in $G$ and let $N=\big<v_2,v_3,...,v_7\big>_G$ denote the subgraph induced by them. Let $H=\big<v_8,...,v_{12}\big>_G$ and let $L$ denote the subset of $E(G)$ of edges having one endpoint in $N$ and the other in $H$. Then, as before, since the degree of every vertex in $G$ is 6, we have:
\[2|E(H)|+|L|=30\]
\[2|E(N)|+|L|=30,\] thus $|E(N)|=|E(H)|$. If $|E(H)|=10$, then $H \simeq K_5$. Since $\delta(G)= 6$ by hypothesis, each vertex in $H$ must be adjacent to a vertex in $N$. Contracting the edges $v_1v_i$, for $2\le i \le 7$, produces a $K_6$ minor of $G$. It follows that $\lvert E(N) \rvert \le 9$ and, unless $N$ is 3-regular, there exists at least a vertex of $N$ which has at most two neighbors in $N$. In \cite{Jo}, J{\o}rgensen proved that in a 6-regular graph, if the open neighborhood of every vertex of $G$ is 3-regular, then any connected component of the graph is isomorphic to either $K_{3,3,3}$ or the complement of the Petersen graph. Since both contain $K_6$ minors, it suffices to consider the case $deg_N(v_i)\le 2$, for some $2\le i \le 7$.\\

If, for some $2\le i \le 7$, $deg_N(v_i)=0$, then contracting the edge $v_1v_i$ produces a minor of $G$ of order 11 and size 35. By Theorem \ref{Mader}, this minor has a $K_6$ minor.\\

If, for some $2\le i \le 7$, $deg_N(v_i)=1$, then contracting the edge $v_1v_i$ produces a minor $M$ of $G$ of order 11 and size 34. Furthermore, the degree sequence of this minor would be $(5,6,6,6,6,6,6,6,6,6,9)$, hence, by Lemma 4, $M$ cannot be apex. By Lemma \ref{11cockade}, $M$ would have a $K_6$ minor.\\

We may then assume that, for $2\le i \le 7$, $deg_N(v_i) \ge 2$ and, w.l.o.g., the neighbors of $v_2$ in $N$ are $v_3$ and $v_4$. \\

{\bf Subcase 3.1} Assume that $v_3v_4\in E(G)$. Contracting the edge $v_1v_2$ produces a minor $M$ of $G$ of order 11 and size $33$. Furthermore, the degree sequence of this minor is $(5,5,6,6,6,6,6,6,6,6,8)$. Via a relabelling, we may assume that $deg_M(v_1)=deg_M(v_2)=5$, and that $deg_M(v_6)=8$. As before, let $N=\big<v_2,v_3,v_4,v_5,v_6\big>_M$ be the subgraph of $M$ induced by all the neighbors of $v_1$, $H=\big<v_7,...,v_{11}\big>_M$, and $L$ the subset of edges of $M$ with one endpoint in $N$ and the other in $H$. Adding degrees we get:
 
\[2|E(H)|+|L|=30\]
\[2|E(N)|+|L|=26,\]

thus $|E(H)|=|E(N)|+2$. Note that $|E(H)| \le 8$, since if $|E(H)|=10$, contracting $v_1v_i$ for $2\le i \le 6$, produces a $K_6$ minor of $M$; if $|E(H)|=9$, that is $H\simeq K_5^-$, then, w.l.o.g., assume $v_7v_8\notin E(H)$. Since $deg_M(v_7)=deg_M(v_8)=6$, it follows that both $v_7$ and $v_8$ have each three neighbors among the five vertices of $N$, thus they share a common neighbor in $N$, say $v_i$. Contracting the edges $v_iv_7$ and $v_1v_j$, for $2\le j\ne i \le 6$, we obtain a $K_6$ minor of $M$. \\

We may then assume that $|E(H)|\le 8$, and thus $|E(N)|\le 6$. If any vertex $v_i$ of $N$ has at most one neighbor in $N$, contracting the edge $v_1v_i$ would produce a minor of $M$ of order 10 and size at least 31. By Theorem \ref{Mader}, this would contain a $K_6$ minor. So every vertex of $N$ has at least two neighbors in $N$. If follows that $5\le |E(N)|\le 6$ and, by \cite{atlas}, $N$ is isomorphic to one of the four graphs in Figure \ref{ABCD}.

\begin{figure}[htpb!]
\begin{center}
\begin{picture}(350, 80)
\put(0,0){\includegraphics[width=5in]{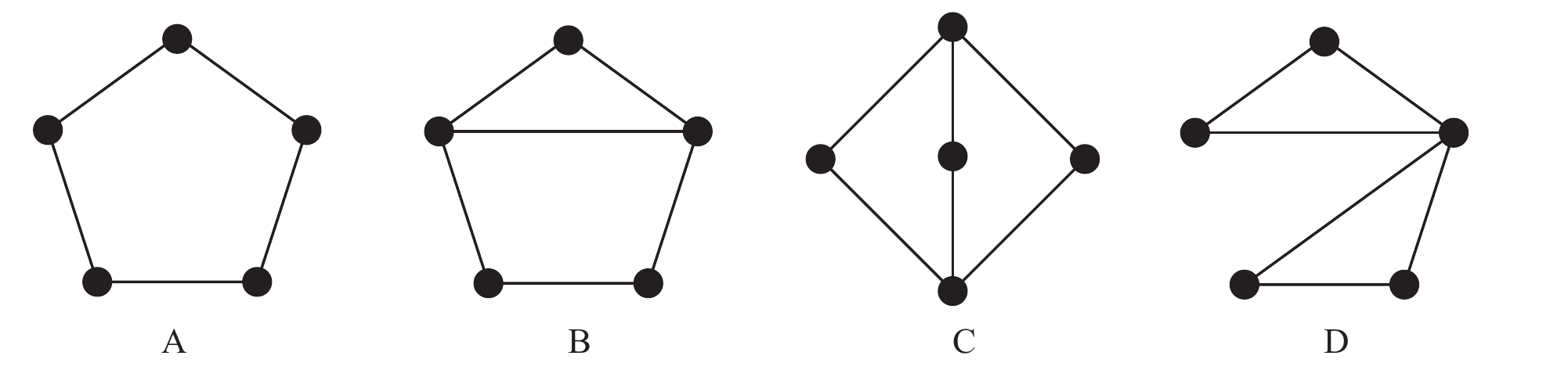}}
\end{picture}
\caption{Graphs of order 5 and size at most 6, with minimum degree 2.}
\label{ABCD}
\end{center}
\end{figure} 

If $deg_N(v_2)=2$, then contracting the edge $v_1v_2$ produces an order 10 and size 30 minor of $M$, with degree sequence $(5,6,6,6,6,6,6,6,6,7)$. The following lemma shows that $M$ contains a $K_6$ minor.

\begin{lemma}
Let $M$ denote a graph of order 10 and size 30, such that $\delta(M)\ge 5$. Assume that $M$  is not apex and has at most five vertices of degree 5. Then $M$ contains a $K_6$ minor.
\label{10cockade} 
\end{lemma}

\begin{proof}
	By Theorem \ref{cockade}, either $M$ contains a $K_6$ minor or is a $MP_1$-cockade. Since $M$ is not apex, it follows that $M$ is the clique sum over $S\simeq K_4$ of two $MP_1$-cockades. Since $\delta(M)\ge5$, any connected component of $M-S$ is at least size two. If any of the connected components of $M-S$ has exactly two vertices, then that component together with $S$ induce a $K_6$ subgraph of $M$. The only situation left to discuss is when $M-S$ has exactly two size 3 connected components, $Q_1$ and $Q_2$. At least one of them, say $Q_1$, contains at most 2 vertices of degree 5. If we denote by $L'$ the set of edges connecting the vertices of $Q_1$ to the vertices of $S$, then 
\[2|E(Q_1)|+|L'|\ge 6+5+5=16,\]
 hence
 \[|E(Q_1)|+|L'|\ge 13 \Rightarrow |E(Q_1)|+|L'|+|E(S)|\ge 19,\] thus $Q_1$ and $S$ induce a subgraph of $M$ of order 7 and size 19. By Theorem \ref{Mader}, this subgraph contains a $K_6$ minor.
 \end{proof}
 
 If $deg_N(v_2)=3$, then $N$ is isomorphic to either graph $B$ or graph $C$ in Figure \ref{ABCD}. Furthermore, $v_2$ neighbors in $N$ a vertex (say $v_3$) of total degree 6 which has degree 2 in $N$ and does not neighbor the vertex of degree 8. Contracting the edge $v_1v_3$ produces a minor $P$ of $M$ of order 10 and size 30. Furthermore, the degree sequences of this minor is $(4,5,6,6,6,6,6,6,7,8)$ and thus it is not apex. By Theorem \ref{cockade}, $P$ either contains a $K_6$ minor or it is the clique sum over $S\simeq K_4$ of two $MP_1$-cockades. Every vertex of $S$ has degree at least 5 since it must connect to every connected component of $P - S$. Let $Q_1$ denote the connected component of $P\backslash S$ which contains the vertex of degree 4. Let $H$ denote the graph induced by the the vertices of $P-(S\cup Q_1)$ and let $L''$ denote the set of edges of $P$ with one endpoint in $S$ and the other in $H$. If $|E(Q_1)|=1$, then 
\[|L''|+|E(H)|=20\]
\[|L''|+2|E(H)|\ge 6+6+6+6+5=29.\] It follows that $|E(H)|\ge9$, that is $H\simeq K_5^-$ or $H\simeq K_5$ and a sequence of contractions in $P$ will create a $K_6$ minor. 
If $|E(Q_1|=2$, then $\big<Q_1\cup S\big>_P\simeq K_6^-$ and
\[|L''|+|E(H)|=16\]
\[|L''|+2|E(H)|\ge 6+6+6+6=24.\]  It follows that $|E(H)|\ge 8$, which is a contradiction ($H$ has only 4 vertices). It follows that $|E(Q)|=3$ and that $\big<S\cup H\big>_P$ contains a $K_7^-$ minor.\\

If $deg_N(v_2)=4$, then $N\simeq D$ of Figure \ref{ABCD} and contracting the edge connecting $v_1$ to the common neighbor of $v_2$ and $v_6$ in $N$ produces a minor $P$ of $M$ of order 10 and size 30, with degree sequence $(4,6,6,6,6,6,6,6,7,7)$, where the two degree 7 vertices and the degree 4 vertex form a triangle. Theorem \ref{cockade} shows that $P$ is a clique sum over $K_4$ of two $MP_1$-cockades. Furthermore $P-K_4$ has exactly two connected components, $Q_1$ and $Q_2$. As any vertex that's part of the clique has at least degree 5 in $P$, we may assume that the vertex of degree 4 is a vertex of $Q_1$. Unless $|Q_1|=1$, both $Q_1$ and $Q_2$ will contain vertices of degree at least six in $P$, hence $|Q_1|=|Q_2|=3$. But this implies that contracting any edge incident to the vertex of degree 4 in $Q_1$ produces a $K_6$ minor of the graph induced by $Q_1\sqcup K_4$. 

If $|Q_1|=1$, then let $L''$ denote the set of edges in $P$ with one endpoint in $K_4$ and the other in $Q_2$. It follows that
\[|L'|=7+7+6+6-12-4=10\]
\[|E(Q_2)|=10,\] hence $Q_2\simeq K_5$ and thus contracting the subgraph induced by $Q_1\sqcup K_4$ to a point produces a $K_6$ minor.\\

{\bf Subcase 3.2} Assume that $v_3v_4\notin E(G)$. Contracting the edge $v_1v_2$ produces a minor $M$ of $G$ of order 11 and size $33$. Furthermore, the degree sequence of this minor is $(5,5,6,6,6,6,6,6,6,6,8)$. Via a relabelling, we may assume that $deg_M(v_1)=deg_M(v_7)=5$, and that $deg_M(v_6)=8$. As before, let $N=\big<v_2,v_3,v_4,v_5,v_6\big>_M$ be the subgraph of $M$ induced by all the neighbors of $v_1$, $H=\big<v_7,...,v_{11}\big>_M$, and $L$ the subset of edges of $M$ with one endpoint in $N$ and the other in $H$. Adding degrees we get:
\[2|E(H)|+|L|=29\]
\[2|E(N)|+|L|=27,\]
thus $|E(H)|=|E(N)|+1$. Furthermore, if $|E(H)|=10$, that is $H\simeq K_5$, contracting $v_1v_i$ for $2\le i \le 6$ produces a $K_6$ minor. 

Assume $|E(H)|=9$ and $|E(N)|=8$. Then by \cite{atlas}, $N$ is isomorphic to one of the two graphs in Figure \ref{AB}.

\begin{figure}[htpb!]
\begin{center}
\begin{picture}(250, 100)
\put(0,0){\includegraphics[width=3in]{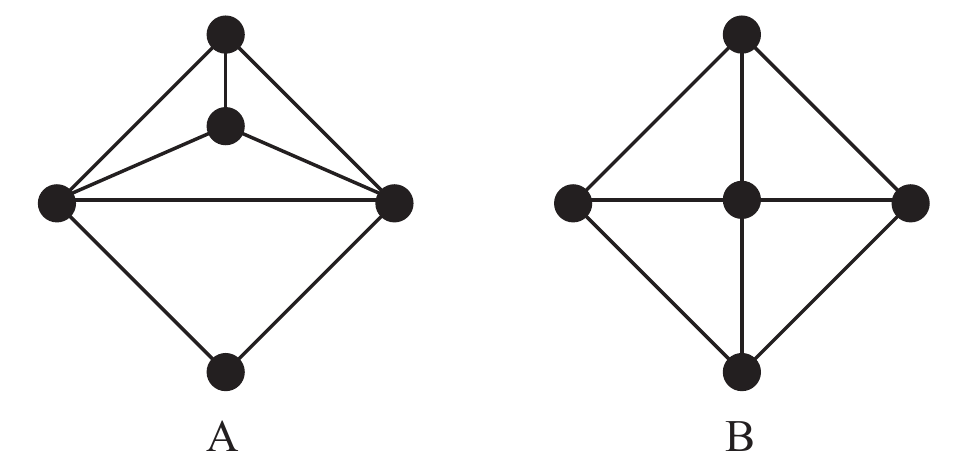}}
\end{picture}
\caption{Graphs of order 5 and size 8, with minimum degree 2.}
\label{AB}
\end{center}
\end{figure} 

If $N\simeq A$ in Figure \ref{AB}, as all vertices of $N$ have minimum degree 6, contracting the vertices of $H$ (which is connected) to a single point and then further contracting the edge joining the newly obtained point and the only vertex of degree 2 in $N$ produces a $K_6$ minor. 

If $N\simeq B$ in Figure \ref{AB}, then we distinguish four cases based on the position of $v_6$ and $v_7$ inside $N$ and $H$, respectively.

\begin{figure}[htpb!]
\begin{center}
\begin{picture}(250, 100)
\put(0,0){\includegraphics[width=3.2in]{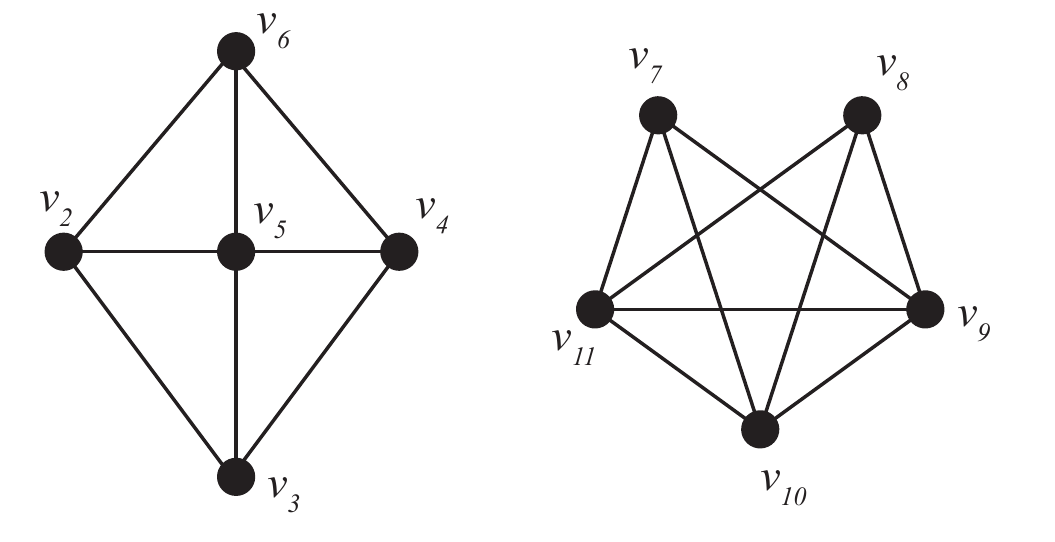}}
\end{picture}
\caption{Case 1: $deg_N(v_6)=3$ and $deg_H(v_7)=3$.}
\label{case1}
\end{center}
\end{figure} 

{\bf Case 1} Assume $deg_N(v_6)=3$ and $deg_H(v_7)=3$. W.l.o.g., assume $deg_H(v_8)=3$, that is $v_7v_8$ is the only edge missing in the complete graph on the vertices of $H$. If $v_8$ neighbors $v_6$, contracting the edge $v_6v_8$ and all the edges of $\big<v_1,v_2,v_3,v_4,v_5\big>_M$ produces a $K_6$ minor of $M$. If $v_6$ does not neighbor $v_8$, then contracting $\big<v_1,v_2,v_3,v_4,v_5,v_8\big>_M$ to a point produces a $K_6$ minor of $M$.

\begin{figure}[htpb!]
\begin{center}
\begin{picture}(250, 110)
\put(0,0){\includegraphics[width=3.2in]{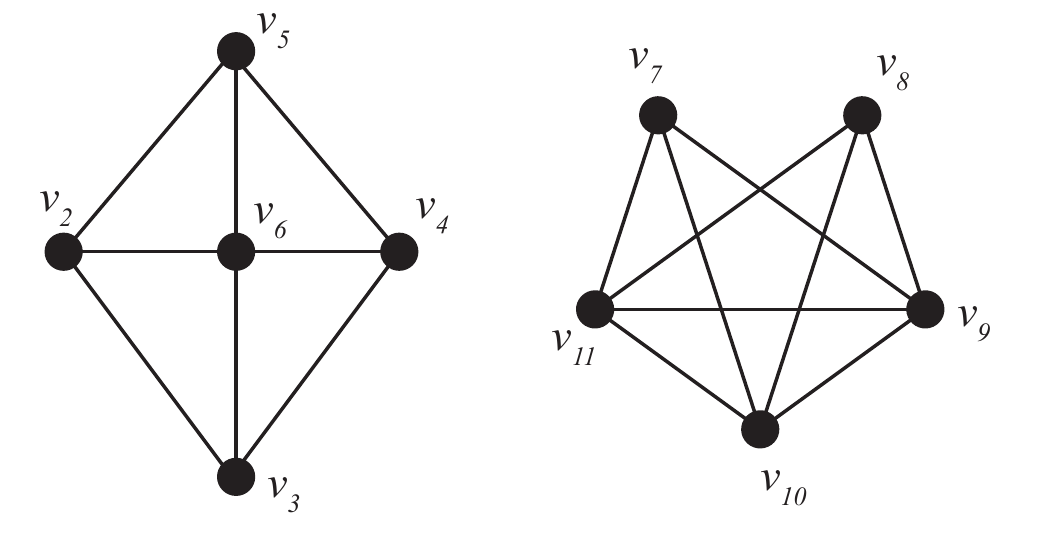}}
\end{picture}
\caption{Case 2: $deg_N(v_6)=4$ and $deg_H(v_7)=3$.}
\label{case2}
\end{center}
\end{figure} 

{\bf Case 2} Assume $deg_N(v_6)=4$ and $deg_H(v_7)=3$. W.l.o.g., we may assume $deg_H(v_8)=3$. If $v_7$ and $v_8$ share a neighbor in $N$, say $v_j$, then contracting $v_jv_7$ and then all the edges of $\big<v_1,N- v_j\big>_M$, we obtain a $K_6$ minor of $G$. So, as a set, $\{v_7, v_8\}$ neighbors all the vertices of $N$. W.l.o.g., $v_5v_7, v_6v_7,v_5v_9 \in E(M)$, $v_3v_5 \notin E(N)$. If $v_3$ neighbors either of $v_{10}$ or $v_{11}$, say $v_{10}$, then contracting the edges $v_3v_{10}, v_5v_9$ will connect $v_3$ and $v_5$, contracting all the edges of $\big<v_2,v_7,v_8,v_{11}\big>_M$ will connect $v_2$ and $v_4$, and thus we obtain a $K_6$ minor. But then, $v_3$ must neighbor $v_9$, and contracting the edge $v_3v_9$ and all the edges of $\big<v_7,v_8,v_{10},v_{11},v_2\big>_M$ produces a $K_6$ minor.

\begin{figure}[htpb!]
\begin{center}
\begin{picture}(250, 110)
\put(0,0){\includegraphics[width=3.2in]{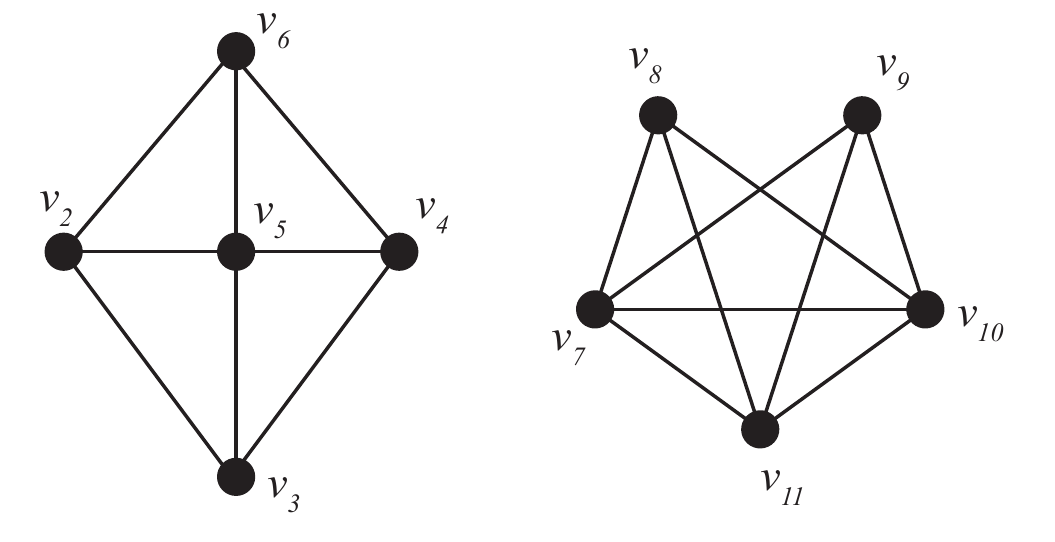}}
\end{picture}
\caption{Case 3: $deg_N(v_6)=3$ and $deg_H(v_7)=4$.}
\label{case3}
\end{center}
\end{figure} 

{\bf Case 3} Assume $deg_N(v_6)=3$ and $deg_H(v_7)=4$. W.l.o.g., we may assume $deg_H(v_8)=deg_H(v_9)=3$, and that $deg_N(v_5)=4$. Since $v_8$ and $v_9$ each connect to three vertices of $N$, they have a common vertex in $N$. If this vertex is not $v_6$, say $v_i$, then contacting the edges $v_iv_8$ and $v_1v_j$ for $2\le j \ne i \le 6$, we obtain a $K_6$ minor of $M$. So $\{v_7, v_8, v_9\}$, as a set, neighbors all the vertices of $N$. Since $v_3$ does not neighbor $v_7$ and cannot neighbor both $v_8$ and $v_9$, it must neighbor one of $v_{10}$ or $v_{11}$. But then, contracting $v_6, v_{10}$ and $v_{11}$ to a point and then contracting $v_7, v_8, v_9,v_2$ to another point, we obtain a $K_6$ minor of M.\\

\begin{figure}[htpb!]
\begin{center}
\begin{picture}(250, 110)
\put(0,0){\includegraphics[width=3.2in]{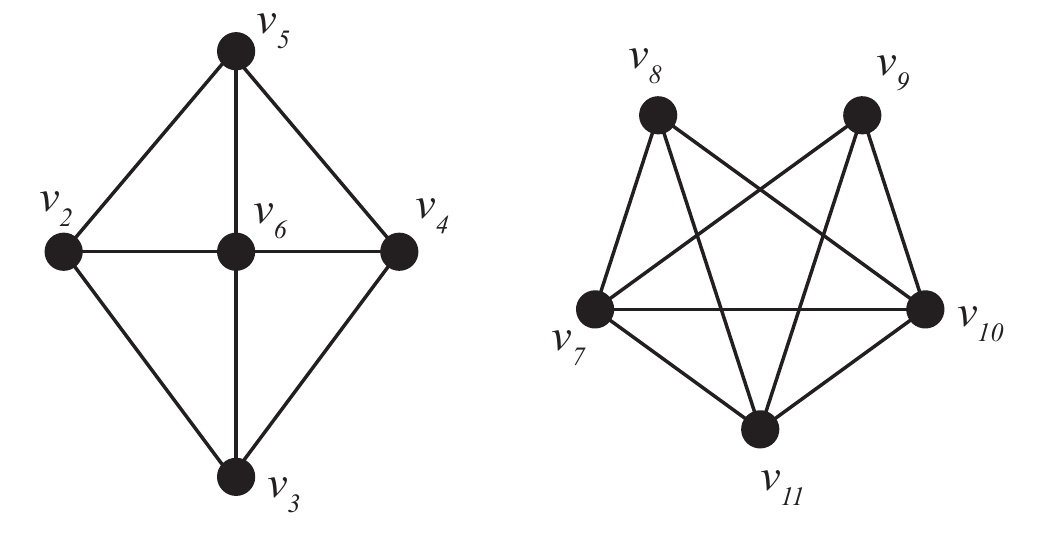}}
\end{picture}
\caption{Case 4: $deg_N(v_6)=4$ and $deg_H(v_7)=4$.}
\label{case4}
\end{center}
\end{figure}

{\bf Case 4} Assume $deg_N(v_6)=4$ and $deg_H(v_7)=4$. W.l.o.g., we may assume that $deg_H(v_8)=deg_H(v_9)=3$. Since both $v_8$ and $v_9$ connect to three vertices of $N$, they must share at least one common neighbor in $N$. If that common neighbor is not $v_6$, say $v_2$, contracting the edges $v_2v_8$, and $v_1v_i$ for $3\le i \le 6$, we obtain a $K_6$ minor of $M$. It follows that $v_6$ connects to $v_7,v_8$ and $v_9$ and that, as a set $\{v_8,v_9\}$ neighbor all the vertices of $N$. If $\{v_{10}v_{11}\}$ neighbors, as a set, any non neighbors in $N$, say $v_2$ and $v_4$, then contracting the edges $v_{10}v_{11}$,$v_2v_{10}$, and all the edges of $\big<v_5,v_7,v_8,v_9\big>_M$, we obtain a $K_6$ minor of $M$. If not, as $v_6$ neighbors neither $v_{10}$ nor $v_{11}$, $v_{10}$ and $v_{11}$ must both neighbor an edge of $N$ not incident to $v_6$, say $v_2v_3$. But since $v_2$ neighbors $v_1,v_3,v_5,v_6,v_{10}$ and $v_{11}$ and one of $v_8$ or $v_9$, it follows that $v_2$ has degree 7 in $M$, a contradiction.\\

It follows that $|E(N)|\le 7$. On the other hand, if any of $v_2,...,v_6$ have degree at most 1 in $N$, contracting the edge connecting that vertex to $v_1$ would produce an order 10 size 31 minor of $M$, which would have a $K_6$ minor by Theorem \ref{Mader}. This shows that $\delta(N)\ge 2$ and that $|E(N)|\ge 5$. Furthermore, if any of the degree 6 neighbors of $v_1$ have degree 2 in $N$, contracting the edge connecting that neighbor to $v_1$ would produce a non-apex graph of order 10 and size 30, with minimal degree at least 5, and at most three vertices of degree 5. By Lemma \ref{10cockade}, this graph would contain a $K_6$ minor. This observation handles the cases $|E(N)|=5$ and $|E(N)|=6$, since any graph on 5 vertices with minimum degree 2 and size at most 6 has at least two vertices of degree exactly 2.

Assume that $|E(N)|=7$, $mindeg(N)=2$ and that $N$ has only one vertex of degree 2. Then $N$ is isomorphic to the graph in Figure \ref{EN7}. Furthermore, $deg_N(v_6)$ is 2, thus $v_6$ connects to all the vertices of $H$. If $v_7$ connects to two or more vertices of $N$, then contracting $v_1v_i$ for $2\le i\le 5$ and $H$ to one of its $K_4$ minors ($H$ has 8 edges and 5 vertices, by Theorem \ref{Mader} it has a $K_4$ minor) we obtain a $K_6$ minor. It follows that $v_7$ connects to $v_8,v_9,v_{10}$ and $v_{11}$. Furthermore, the open neighborhood of $v_7$ contains exactly 8 edges. By symmetry between $v_1$ and $v_7$ and Subcase 3.2, $|E(N)|=8$, it follows that $M$ has a $K_6$ minor. 

\begin{figure}[htpb!]
\begin{center}
\begin{picture}(120, 110)
\put(0,0){\includegraphics[width=1.5in]{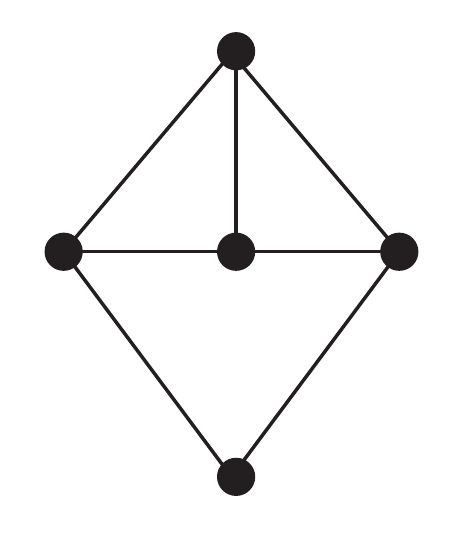}}
\end{picture}
\caption{The unique graph of order 5, size 7, and min degree 2, with exactly one vertex of degree 2.}
\label{EN7}
\end{center}
\end{figure} 

\end{proof}

\section{Future Explorations}

\begin{enumerate}

\item Is it true that any simple graph of order at most 14 and minimum degree at least 6, which is not apex, contains a $K_6$-minor? Note that in the proof of Theorem \ref{main}, we used weaker versions of Lemma \ref{11cockade} and Lemma \ref{10cockade}. Similar lemmas hold for graphs of order 13 and 12, respectively. They provide a first step in generalizing Theorem \ref{main} for graphs of order at most 14. 

\item The result of this paper shows that, for graphs of order 12, weaker assumptions are needed for the conclusion of J{\o}rgensen's conjecture to be true. What is the minimum $n>12$ for which the minimum degree 6 condition is no longer sufficient and the 6-connected condition is needed?  Such $n$ would have to be at most 26, as the following example demonstrates.\\

Let $G_1\simeq G_2 \simeq K_1\ast Ic$, where $Ic$ denotes the icosahedral graph (5-regular, maximal planar, order 12). Let $G$ be the graph obtained by joining $G_1$ and $G_2$ with a single edge. Then $\delta(G)=6$, $G$ is not apex and it has no $K_6$ minor.

\end{enumerate}

\end{document}